\documentclass[reqno,12pt]{amsart}
\usepackage{amsmath,amsthm,enumerate,cite,graphics,pstricks,amsfonts,latexsym,amsopn,verbatim,amscd,amssymb}
\usepackage{color}
\usepackage{psfrag}
\usepackage[all]{xy}

\theoremstyle{plain}
\newtheorem{thm}{Theorem}[section]

\newtheorem{lem}[thm]{Lemma}
\newtheorem{cor}[thm]{Corollary}
\newtheorem{prop}[thm]{Proposition}

\newtheorem{que}[thm]{Question}

\theoremstyle{definition}

\newtheorem{rem}[thm]{Remark}

\DeclareMathOperator{\Aut}{Aut}

\DeclareMathOperator{\Cent}{Cent}

\begin{document} 

\title[On groups with same number of centralizers]{On groups with same number of centralizers} 

\author[S. J. Baishya  ]{Sekhar Jyoti Baishya} 
\address{S. J. Baishya, Department of Mathematics, Pandit Deendayal Upadhyaya Adarsha Mahavidyalaya, Behali, Biswanath-784184, Assam, India.}

\email{sekharnehu@yahoo.com}

\begin{abstract}
In this paper, among other results, we give some sufficient conditions for every non-abelian subgroup of a group to be isoclinic with the group itself. It is also seen that under certain conditions, two groups have same number of element centralizers implies they are isoclinic. We prove that if $G$ is any group having $4, 5, 7$ or $9$ element centralizers and $H$ is any non-abelian subgroup of $G$, then $\mid \Cent(G)\mid=\mid \Cent(H)\mid$ and $ G' \cong  H' \cong C_2, C_3, C_5$ or $C_7$ respectively. Furthermore, it is proved that if $G$ is any group having $n \in \lbrace 4, 5, 6, 7, 9 \rbrace$ element centralizers, then $\mid G' \mid=n-2$.
\end{abstract}

\maketitle

\section{Introduction} \label{S:intro1}

Given any group $G$, let $\Cent(G)$ and $nacent(G)$ denotes, respectively, the set of centralizers and the set of non-abelian centralizers of elements of $G$. A group $G$ is said to be $n$-centralizer if $\mid \Cent(G)\mid=n$.  In 1994 Belcastro and Sherman  \cite{ctc092} introduced  the notion of $n$-centralizer groups and since then the influence of $\Cent(G)$ on the structure of group have been studied extensively. See \cite{nca, ashrafinew,extcent, con, baishyaF, baishyarima} for recent advances on this and related areas. Perhaps motivated by the impact of  $\mid \Cent(G) \mid$ on the group, Ashrafi and Taeri \cite{taeri} in 2005  asked the following question which was disproved by Zarrin \cite{zarrin094}: Let $G$ and $H$ be finite simple groups. Is it true that if $|\Cent(G)|=|\Cent(H)|$, then $G$ is isomorphic to $H$? Amiri and Rostami \cite{rostami1} in 2015  put forward the following analogue question which was also disproved  by Khoramshahi and Zarrin \cite{con}: Let $G$ and $H$ be finite simple groups. Is it true that if $|nacent(G)|=|nacent(H)|$, then $G$ is isomorphic to $H$? In this context we have the following natural question:

\begin{que} \label{q1}
What can be said about the relationship between two groups if they have the same number of element centralizers. 
\end{que}

It may be mentioned here that if an $n$-centralizer group $G$ is isoclinic with a group $H$, then 
$|\Cent(G)|=|\Cent(H)|$ (see \cite{non, rahul2}). However, the converse is not true in general. For example, if $G$ is a non-abelian group of order $27$, then $|\Cent(G)|=|\Cent(S_3)|=5$, but $G$ and $S_3$ are not isoclinic. The authors in \cite{con} studied and obtained some conditions under which the converse of this statement holds. In this paper, we continue with Question \ref{q1} and improve some earlier results. We obtain some sufficient conditions for every non-abelian subgroup of a group to be isoclinic with the group itself. In particular, it is seen that any non-abelian subgroup of a $4$ or $5$-centralizer group is isoclinic with the group itself, which improves \cite[Theorem 3.5]{con}. It is also proved that any two arbitrary $4$-centralizer groups are isoclinic and any two arbitrary nilpotent $5, 7$ or $9$-centralizer groups are isoclinic. We obtain that if $H$ is any non-abelian subgroup of an $n$-centralizer group $G$, where $n=4, 5, 7$ or $9$, then $|\Cent(G)|=|\Cent(H)|$ and $ G' \cong  H' \cong C_2, C_3, C_5$ or $C_7$ respectively. For any subgroup $H$ of an arbitrary $8$-centralizer group $G$, it is observed that $|\Cent(G)|=|\Cent(H)|$ implies $G$ is isoclinic with $H$. Given any $n$-centralizer group $G$ with $n \in \lbrace 4, 5, 6, 7,  9\rbrace$, we see that $\mid G' \mid=n-2$. A finite group is said to be of conjugate type $(m, 1)$ if every proper element centralizer is of index $m$. For any two finite groups $G$ and $H$ of conjugate type $(p, 1)$, $p$ a prime, it is proved that $|\Cent(G)|=|\Cent(H)|$ implies $G$ is isoclinic with $H$. Among other results, we prove that if $G$ is any finite $(n+2)$-centralizer group of conjugate type $(n, 1)$, then $G$ is a CA-group (i.e., every proper element centralizer of $G$ is abelian) and $\frac{G}{Z(G)}$ is elementary abelian of order $n^2$, which improves \cite[Theorem 3.3]{en09}.

Throughout this paper, for a group $G$,  $Z(G)$ and $G'$ denotes its center and commutator subgroup respectively,  $C_G(x)$  denotes the centralizer of $x \in G$ (however, if there is no confusion in the context then we simply write $C(x)$ in place of $C_G(x)$), $C_n$  denotes the cyclic group of order $n$ and $D_{2n}$ denotes the dihedral group of order $2n$. Some results of this paper holds for finite groups only and we have specifically mentioned it whenever necessary.

\section{Definitions and basic results} \label{S:introd}

We begin with the notion of isoclinism between two groups introduced by P. Hall \cite{hall} in 1940. Two groups $G$ and $H$ are said to be isoclinic if there are two isomorphisms $\varphi : G/Z(G) \longrightarrow H/Z(H)$ and $\phi : G' \longrightarrow H'$ such that if 
\[
\varphi(g_1Z(G))=h_1Z(H) \;\; \text{and} \;\; \varphi(g_2Z(G))=h_2Z(H)
\]
with $g_1, g_2 \in G, h_1, h_2 \in H$, then 
\[
\phi([g_1, g_2])=[h_1, h_2].
\] 

Isoclinism is an equivalence relation weaker than isomorphism having many family invariants. Here we list some of the invariants concerning the element centralizers of two isoclinic groups. 

Recall that a group $G$ is called an F-group if every non-central element centralizer contains no other element centralizer and a CA-group if  all non-central element centralizers are abelian. Finite   groups having exactly two class sizes are called I-groups which are  direct product of an abelian group and a group of prime power order \cite {ito}.  Two elements of a group are said to be $z$-equivalent or in the same $z$-class if their centralizers are conjugate in the group. Being $z$-equivalent is an equivalence relation which is weaker than conjugacy relation. A $z$-equivalence class is called a $z$-class.
In the following result, $\omega (G)$ denotes the size of a maximal set of pairwise non-commuting elements of a group $G$. 

\begin{prop}\label{b439}
If an $n$-centralizer group $G$ is isoclinic with a group $H$, then  
\begin{enumerate}
	\item $\omega (G)=\omega(H)$ (\cite[Lemma 2.1]{non}).
	\item $z$-classes in $G$=$z$-classes in $H$ (\cite[Theorem 2.2]{rahul1}).
	\item   $\mid \Cent(G)\mid=\mid \Cent(H)\mid$ (\cite[Lemma 3.2]{non}, \cite[Theorem A]{rahul2}).
	\item $\mid nacent(G)\mid=\mid nacent(H)\mid$.
	\item   $G$ is a CA-group implies $H$ is also a CA-group.
	\item  $G$ is an F-group implies $H$ is also an F-group.
	\item  $G$ is an I-group implies $H$ is also an I-group  (\cite{hall}, \cite[Proposition 2.2]{ish1}).
\end{enumerate} 
\end{prop}

\begin{proof}
d)  Let $\varphi : G/Z(G) \longrightarrow H/Z(H)$  be the isomorphism. Then $\varphi$ induces a bijection between the subgroups of $G$ containing $Z(G)$ and the subgroups of $H$ containing $Z(H)$ and the corresponding subgroups are isoclinic \cite[pp. 134]{hall}.  For any $x \in G$, consider its centralizer $C_G(x)$ which contains $Z(G)$. In view of proof of \cite[Theorem A]{rahul2}, the corresponding subgroup of $H$ containing $Z(H)$ is $C_H(y)$, where $yZ(H)=\varphi(xZ(G))$. Therefore $C_G(x)$ is isoclinic with $C_H(y)$. Hence the result follows.\\

e) It follows from part (d)\\

f) Let $\varphi : G/Z(G) \longrightarrow H/Z(H)$  be the isomorphism. Suppose $H$ is not an F-group. Then $C_H(a) < C_H(b)$ for some $a, b \in H \setminus Z(H)$. Therefore $\frac{C_H(a)}{Z(H)} < \frac{C_H(b)}{Z(H)}$ and consequently, in view of proof of \cite[Theorem A]{rahul2}, we have $\frac{C_G(x)}{Z(G)} < \frac{C_G(y)}{Z(G)}$ for some $x, y \in G \setminus Z(G)$, where $\varphi(\frac{C_G(x)}{Z(G)})=\frac{C_H(a)}{Z(H)}$ and $\varphi(\frac{C_G(y)}{Z(G)})=\frac{C_H(b)}{Z(H)}$. It now follows that $C_G(x) < C_G(y)$, which implies $G$ is not an F-group.
\end{proof}

The following theorems will be used to obtain some of our results. 

\begin{thm}{(p.135 \cite{hall})} \label{ff1}
Every group is isoclinic to a group whose center is contained in the commutator subgroup. 
\end{thm}

\begin{thm}{(\cite{hall}, \cite[Proposition 2.2]{ish1})} \label{ff155}
Let $G$ and $H$ be finite $p$-groups  ($p$ a prime). Suppose $G$ is isoclonic with $H$. Then $G$ and $H$ are groups of the same conjugate type.
\end{thm}

\begin{thm}{(Theorem 11 \cite{rahul2}, Theorem 3.3 \cite {non})} \label{ff3}
The representatives of the families of isoclinic groups with $n$-centralizers ($n \neq 2, 3$) can be chosen to be finite groups.
\end{thm}

For any subgroup $H$ of $G$, it is easy to see that $C_H(x)=C_G(x)\cap H$, for any $x \in H$. This gives the following result:

\begin{lem}\label{b0}
Let $H$ be a subgroup of $G$ such that $H \cap Z(G) \lneq Z(H)$. Then the number of centralizers of $G$ produced by the elements of $H$ is atleast $\mid \Cent(H)\mid+1$.
\end{lem}

\begin{proof}
Clearly, the number of   centralizers of $G$ produced by elements of $H$ is equal to the number of centralizers of $G$ produced by the elements of $H \cap Z(G)$+ the number of centralizers of $G$ produced by the elements of $H \setminus (H \cap Z(G)) \geq 1+\mid \Cent(H)\mid $ (note that elements of $H$ that have the same centralizers in $H$ may have different centralizers in $G$).
\end{proof}

\begin{lem}\label{np101}
Let $G$ be a  finite group and $p$ be a prime. If $G$ has a non-central element of order $p$, then  $\mid \Cent(G)\mid \geq p+2$.  
\end{lem}

\begin{proof}
Let $x \in G \setminus Z(G)$ be an element of order $p$. Let $a \in G \setminus C(x)$.  Clearly, $ax^i \in G \setminus C(x)$ for any $i$. Consider the set $X=\lbrace x, a, ax, ax^2, \dots, ax^{p-1} \rbrace$. Observe that if $ax^iax^j=ax^jax^i$ for some $0 \leq i < j \leq p-1$, then $a \in C(x^{j-i})=C(x)$, a contradiction (noting that $gcd((j-i), o(x))=1$). Therefore $X$ is a set of pairwise non-commuting elements of $G$ and $\mid X \mid=p+1$. Hence $\mid \Cent(G)\mid \geq  p+2$. 
\end{proof}

  For any finite group $G$, the author in \cite[Lemma 3.1]{en09} proved that if $G' \cap Z(G)=\lbrace 1 \rbrace$, then  $\mid \Cent(G)\mid= \mid \Cent(\frac{G}{Z(G)})\mid$. However, for any arbitrary $n$-centralizer group we have the following general result:

\begin{prop}\label{b1}
Let $G$ be any $n$-centralizer group and $N\unlhd G$.  If $N \cap G'=\lbrace 1 \rbrace$, then  $\mid \Cent(G)\mid= \mid \Cent(\frac{G}{N})\mid$.   
\end{prop}

\begin{proof}
In view of \cite[pp. 134]{hall} and Proposition \ref{b439} we have the result.
\end{proof}

In response to a question  raised by  Belcastro and Sherman \cite{ctc092}, namely, whether there exists an $n (\neq 2, 3)$-centralizer group, Ashrafi showed \cite[Proposition 2.1]{en09} that there exists $n$-centralizer groups for $n \neq 2, 3$. In this connection, we have the following result which implies we can say something more than Ashrafi's result. It also improves \cite[Proposition 2.2]{ed09} and \cite[Lemma 2]{actc09}. Furthermore, it improves  \cite[Example 16]{nca}, namely, there exists a $2^r$-centralizer CA-group for every $r>1$. In the following result $ C_n {\rtimes}_\theta C_p$ denotes semidirect product of $C_n$ and $C_p$, where $\theta : C_p \longrightarrow \Aut(C_n)$ is a homomorphism.

\begin{prop}\label{bcc1}
Given any group $G$, suppose $\frac{G}{Z(G)}$ be non-abelian and $p$ be a prime. If $\frac{G}{Z(G)} \cong C_n {\rtimes}_\theta C_p$,
then $G$ is an $(n+2)$-centralizer CA-group.
\end{prop}

\begin{proof}
In view of Theorem \ref{ff3}, $G$ is isoclinic with a finite group. Now, the result follows using Proposition \ref{b439} and \cite[Proposition 2.9 and Lemma 2.10]{baishya}.
\end{proof}

Recall that the generalized quaternion group $Q_{4m}$ has the presentation $\langle a, b \mid a^{2m}=1, b^2=a^m, bab^{-1}=a^{-1} \rangle, m \geq 2 $.

\begin{cor}\label{bcor1}
There exists $n$-centralizer CA-groups for $n \geq 4$.
\end{cor}

\begin{proof}
In view of Proposition \ref{bcc1}, $Q_{4(n-2)}, n \geq 4$ is an $n$-centralizer CA-group by noting that $\frac{Q_{4(n-2)}}{Z(Q_{4(n-2)})} \cong D_{2(n-2)}$ for any $n \geq 4$.
\end{proof}

\begin{rem}\label{rem111}
Let $p$ be a prime. A finite $p$-group $G$ is said to be a special $p$-group of rank $k$ if $G'=Z(G)$ is elementary abelian of order $p^k$ and $\frac{G}{G'}$ is elementary abelian. Furthermore, a finite group $G$ is extraspecial if $G$ is a special $p$-group and $\mid G' \mid=\mid Z(G) \mid=p$.  It is well known that every extraspecial $p$-group has order $p^{2a+1}$ for some positive integer $a$. Furthermore, for  every prime $p$ and every positive integer $a$, there exists, upto isomorphism, exactly two extraspecial groups of order $p^{2a+1}$. Moreover, any two extraspecial groups of same order are isoclinic (see \cite[pp. 7]{lewis}). Again, if $G$ is any group and $A$ is an abelian group, then $G$ and $G \times A$ are isoclinic (see \cite[pp. 135]{hall}).
 \end{rem}

In this context we have the following result.

\begin{prop}\label{bcor2}
There exists $2^{2n}$-centralizer F-groups which are not CA-groups for $n>1$.
\end{prop}

\begin{proof}
Let $G$ be an extraspecial group of order $2^{2n+1}, n>1$. Then in view of \cite[Proposition 2.26]{baishyarima} and \cite[Proposition 3.13]{baishyaF}, $G$ is an $2^{2n}$-centralizer F-group which is not a CA-group.
\end{proof}

\section{Main results} \label{S:main}

The following key result helps in determining whether a given group is CA or not.

\begin{prop}\label{b011}
An arbitrary group $G$ is a CA-group if and only if  $Z(H)=Z(G) \cap H$ for any non-abelian subgroup $H$ of $G$. In particular,  $ \frac{H}{Z(H)}= \frac{H}{Z(G) \cap H} \cong \frac{HZ(G)}{Z(G)} \leq \frac{G}{Z(G)}$ for any non-abelian subgroup $H$ of a CA-group $G$.
\end{prop}

\begin{proof}
Let $G$ be a CA-group and $H$ be a non-abelian subgroup of $G$. If $a, b \in H$ be such that $ab \neq ba$, then $a, b \in H \setminus ( Z(G) \cup Z(H))$. It is easy to verify that $Z(H)= C_H(a) \cap C_H(b)=C_G(a) \cap H \cap C_G(b) \cap H=Z(G) \cap H$. Conversely, if $Z(H)=Z(G) \cap H$ for any non-abelian subgroup $H$ of $G$, then $G$ is a CA-group. For if 
$C_G(x)$ is non-abelian for some $x \in G \setminus Z(G)$, then $C_G(x) \cap Z(G)=Z(G) \subsetneq Z(C_G(x))$, which is a contradiction. Last part is trivial.
\end{proof}

\begin{cor}\label{b3}
If $ Z(G') = \lbrace 1 \rbrace$ for any CA-group $G$, then  $G$ is isoclinic with $\frac{G}{Z(G)}$. In particular, if $G$ is $n$-centralizer, then $\frac{G}{Z(G)}$ is also an $n$-centralizer CA-group.
\end{cor}

\begin{proof}
Using Proposition \ref{b011}, \cite[pp. 134]{hall} and Proposition \ref{b439} we have the result.
\end{proof}

As an application of Proposition \ref{b011}, we also have the following result.

\begin{prop}\label{p2}
If  $H$ is a non-abelian subgroup of $G$ with $\frac{G}{Z(G)} \cong D_{2n}$, then
\begin{enumerate}
	\item $G$ is an $(n+2)$-centralizer CA-group.
	\item $\frac{H}{Z(H)} \cong D_{2n/d}$ for some divisor $d$ of $n$.
	\item $\mid \Cent(G)\mid=\mid \Cent(H)\mid$ implies $G$ is isoclinic with $H$.
	\end{enumerate} 
\end{prop}

\begin{proof}

a) It follows from Proposition \ref{bcc1}.\\

b) By part (a)   $G$ is a CA-group  and consequently, using Proposition \ref{b011} we have $ \frac{H}{Z(H)}= \frac{H}{Z(G) \cap H} \cong \frac{HZ(G)}{Z(G)} \leq \frac{G}{Z(G)}$. Now,  using \cite[Theorem 3.1]{conrad} we have the result.\\

(c) In view of  part (b),  $\frac{H}{Z(H)} \cong \frac{HZ(G)}{Z(G)} \cong D_{2n/d}$ for $d \mid n$. Now, if $\mid \Cent(H)\mid=n+2$, then by part (a), we have $HZ(G)=G$ and  consequently $G$ is isoclinic with $H$ by \cite[Lemma 2.7]{pL95}.
\end{proof}

Let $H$ be a subgroup of $G$. The author in \cite[Lemma 2.7]{pL95} proved that if $G=HZ(G)$, then $G$ and $H$ are isoclinic, and if $H$ is finite then the converse is also true. We have the following general result for an arbitrary $n$-centralizer group.

\begin{prop}\label{bp4}
Let $G$ be any $n$-centralizer group and $H \leq G$. Then  $G=HZ(G)$ iff $G$ is isoclinic with $H$.
\end{prop}

\begin{proof}
If $G=HZ(G)$, then by  \cite[Lemma 2.7]{pL95}, we have $G$ is isoclinic with $H$.  Conversely, if $G$ is isoclinic with $H$, then  $\mid \Cent(G)\mid=\mid \Cent(H)\mid$ by Proposition \ref{b439}. Therefore using Lemma \ref{b0}, we have $Z(H)=H \cap Z(G)$, which implies  $\frac{H}{Z(H)}= \frac{H}{Z(G) \cap H} = \frac{HZ(G)}{Z(G)} \cong \frac{G}{Z(G)}$ and thus $G=HZ(G)$.
\end{proof}

Using arguments similar to Proposition \ref{bp4}, we also have the following result:

\begin{prop}\label{b4}
Let $G$ be any $n$-centralizer group and $H \leq G$ be such that $\frac{G}{Z(G)} \cong \frac{H}{Z(H)}$. Then  $\mid \Cent(G)\mid=\mid \Cent(H)\mid$ iff $G$ is isoclinic with $H$.
\end{prop}

For groups $G_1$ and $G_2$  in \cite[pp.56]{ed09}, we have $\frac{G_1}{Z(G_1)} \cong \frac{G_2}{Z(G_2)} \cong C_2 \times C_2 \times C_2$, $\mid \Cent(G_1)\mid = 6 $ and $ \mid \Cent(G_2)\mid=8$; which implies $G_1$ and $G_2$ are not isoclinic by Proposition \ref{b439}.  However, for some special situations we have the following result:

\begin{prop}\label{bs4}
Let $G$ be any $n$-centralizer CA-group and $H \leq G$. Then $\frac{G}{Z(G)} \cong \frac{H}{Z(H)}$ iff $G$ is isoclinic with $H$. In particular, $\mid \Cent(G)\mid=\mid \Cent(H)\mid$.
\end{prop}

\begin{proof}
In view of Proposition \ref{b011},  $Z(H)=Z(G) \cap H$ and so $\frac{G}{Z(G)} \cong \frac{H}{Z(H)}= \frac{H}{Z(G) \cap H} = \frac{HZ(G)}{Z(G)}$. Therefore $G=HZ(G)$, and hence $G$ is isoclinic with $H$ by \cite[Lemma 2.7]{pL95}. Converse is trivial. Last part follows from Proposition \ref{b439}.
\end{proof}

\begin{thm}\label{b57}
If  $H$ is a non-abelian subgroup of $G$ with $\mid \frac{G}{Z(G)}\mid=pq$ ($p \leq q$ are primes), then  $G$ is isoclinic with $H$. In particular, $\mid \Cent(G)\mid=\mid \Cent(H)\mid=q+2$.
\end{thm}

\begin{proof}
In view of  Proposition \ref{b011}, $\frac{H}{Z(H)}=\frac{H}{Z(G) \cap H} \cong \frac{HZ(G)}{Z(G)} \leq \frac{G}{Z(G)}$ and so $HZ(G)=G$ by noting that $G$ is a CA-group. Now, using  \cite[Lemma 2.7]{pL95}, $G$ is isoclinic with $H$. Last part follows using Theorem \ref{ff3}, Proposition \ref{b439} and \cite[Corollary 2.5]{baishya}.
\end{proof}

The following result will be used in the next theorem.

\begin{prop}\label{CG17}
Let $G$ be a finite group such that $\frac{G}{Z(G)}=\frac{K}{Z(G)} \rtimes \frac{H}{Z(G)}$ is a Frobenius group with $K$ and $H$ abelian. Then $\mid \Cent(G)\mid= \mid G' \mid+2$. 
\end{prop}

\begin{proof}
Using the third isomorphic theorem, we get $\frac{G}{K} \cong \frac{H}{Z(G)}$. Consequently,   we have $K$ is an abelian normal subgroup of $G$ such that $\frac{G}{K}$ is cyclic. In the present scenario, in view of \cite[Lemma 12.12]{isaacs}, $\mid K \mid=\mid G' \mid \mid K \cap Z(G) \mid$ which forces $\mid \frac{K}{Z(G)}\mid=\mid G' \mid$. Therefore by \cite[Proposition 3.1]{amiri2019}, $\mid \Cent(G)\mid= \mid G' \mid+2$.
\end{proof}

In the following result, which improves \cite[Theorem 3.5]{con}, $(C_m, C_n)$ denotes the Frobenius group with complement $C_m$ and kernel $C_n$.

\begin{thm}\label{c1}
Let $H$ be a non-abelian subgroup of an $n$-centralizer group $G$.

\begin{enumerate}
	\item If $n=4$ or $5$, then $G$ is isoclinic with $H$.
	\item If $n=4, 5, 7$ or $9$, then $\mid \Cent(G)\mid=\mid \Cent(H)\mid$ and $ G' \cong  H' \cong C_2, C_3, C_5$ or $C_7$ respectively.
	\item If $n=8$, then $\mid \Cent(G)\mid=\mid \Cent(H)\mid$ implies $G$ is isoclinic with $H$.
\end{enumerate}
\end{thm}

\begin{proof}

a) We have $\mid \frac{G}{Z(G)}\mid=4, 6$ or $9$ by \cite[Theorem 3.5]{non} and hence the result follows using Theorem \ref{b57}.\\

b) Following \cite[Theorem 3.5]{non} and applying similar arguments to \cite[Theorem 2.6]{baishya1} 
, we have $\frac{G}{Z(G)} \cong (C_4, C_5)$, $(C_6, C_7)$ or $\mid \frac{G}{Z(G)}\mid \in \lbrace 4, 6,9, 10, 14, 21, 25, 49 \rbrace$. In the present scenario, using \cite[Lemma 2.1]{baishya2} and Proposition \ref{b011},  we have $Z(H)=Z(G) \cap H$ and hence $\frac{H}{Z(H)}=\frac{H}{Z(G) \cap H} \cong \frac{HZ(G)}{Z(G)} \leq \frac{G}{Z(G)}$.

Now, if  $\mid \frac{G}{Z(G)}\mid \in \lbrace 4, 6,9, 10, 14, 21, 25, 49 \rbrace$, then using Theorem \ref{b57} we have $\mid \Cent(G)\mid=\mid \Cent(H)\mid$.

Next, suppose $\frac{G}{Z(G)} \cong (C_4, C_5)$.  If $\frac{HZ(G)}{Z(G)}   < \frac{G}{Z(G)}$, then $\mid \frac{H}{Z(H)}\mid=10$ (by noting that $2$-Sylow subgroup of $\frac{G}{Z(G)}$ is cyclic) and so $\mid \Cent(G)\mid=\mid \Cent(H)\mid$ using \cite[Theorem 3.5]{non}. On the other hand if $\frac{HZ(G)}{Z(G)}  = \frac{G}{Z(G)}$, then $HZ(G)=G$ and hence $\mid \Cent(G)\mid=\mid \Cent(H)\mid$  using Proposition \ref{bp4} and  Proposition \ref{b439}.

Finally, suppose $\frac{G}{Z(G)} \cong (C_6, C_7)$. If  $\frac{HZ(G)}{Z(G)}   < \frac{G}{Z(G)}$, then $\mid \frac{H}{Z(H)} \mid=14$ or $21$ by noting that $\frac{G}{Z(G)}$ cannot have a non-abelian subgroup of order $6$ and consequently, applying arguments of \cite[Theorem 3.5]{non} to \cite[Theorem 2.6]{baishya1} we have $\mid \Cent(G)\mid=\mid \Cent(H)\mid$. On the other hand if $\frac{HZ(G)}{Z(G)}  = \frac{G}{Z(G)}$, then $HZ(G)=G$ and hence $\mid \Cent(G)\mid=\mid \Cent(H)\mid$  using Proposition \ref{bp4} and  Proposition \ref{b439}.

Second part follows using Proposition \ref{b439}, Theorem \ref{ff3}, \cite[Theorem 2.3]{baishya} and Proposition \ref{CG17}.\\

c) In view of \cite[Theorem 3.5]{non}, we have $\mid \frac{G}{Z(G)}\mid=8$ or $12$.  In the present scenario, by \cite[Lemma 2.1]{baishya2} and Proposition \ref{b011},  we have  $\frac{H}{Z(H)}=\frac{H}{Z(G) \cap H} \cong \frac{HZ(G)}{Z(G)} \leq \frac{G}{Z(G)}$. Now using  \cite[Theorem 3.5]{non} again, we have $HZ(G)=G$ and so by Proposition \ref{bp4}, $G$ is isoclinic with $H$.
\end{proof}

Note that the $6$-centralizer group $Q_{16}$ has a $4$-centralizer subgroup, namely $Q_8$, which is not isoclinic with $Q_{16}$.  The $7$-centralizer group $(C_4, C_5)$ has a $7$-centralizer subgroup of order $10$ which is not isoclinic with $(C_4, C_5)$. The  $8$-centralizer group $Q_{24}$ has a $4$-centralizer subgroup, namely $Q_8$. Again, the $9$-centralizer group $(C_6, C_7)$ has a $9$-centralizer subgroup of order $21$ which is not isoclinic with  $(C_6, C_7)$.

A finite $p$-group ($p$ a prime) $G$ is semi-extraspecial if for every maximal subgroup $N$ in $Z(G)$ the quotient $\frac{G}{N}$ is extraspecial. It is known that every semi-extraspecial $p$-group is special. Furthermore, a group $G$ is said to be ultraspecial if $G$ is semi-extraspecial and $\mid G' \mid= \sqrt{\mid G: G' \mid}$.  

\begin{prop}\label{ppp1}
If $G$ is an $n$-centralizer group with $n \in \lbrace 4, 5, 6, 7,  9\rbrace$, then $\mid G' \mid=n-2$.
\end{prop}

\begin{proof}

In view of Proposition \ref{b439} and Theorem \ref{ff3} without any loss we may assume that $G$ is a finite group.

Now, suppose $n=6$. Using \cite[Theorem 3.5]{non}, we have $\frac{G}{Z(G)} \cong D_8, A_4, C_2 \times C_2 \times C_2$ or $C_2 \times C_2 \times C_2 \times C_2$. If $\mid \frac{G}{Z(G)}\mid=8$, then in view of \cite[Proposition 2.14]{baishya2}, $G$ has an abelian  centralizer of index $2$ and consequently, using \cite[Theorem 2.3]{baishya}, we have $\mid G' \mid=4$. Again, if $\frac{G}{Z(G)} \cong A_4$, then by \cite[Proposition 2.12]{baishya2},  $G$ has an abelian normal centralizer of index $3$  and consequently, using \cite[Theorem 2.3]{baishya}, we have $\mid G' \mid=4$. Finally, if  $\mid \frac{G}{Z(G)}\mid=16$, then in view of \cite[Proposition 3.21]{baishyaF}, $G$ is isoclinic with an ultraspecial  group of order $64$ and hence $\mid G' \mid=4$. Now, the result follows using Theorem \ref{c1}.
\end{proof}

 Note that for the group $G_2$  in \cite[pp.56]{ed09}, we have $\frac{G_2}{Z(G_2)} \cong C_2 \times C_2 \times C_2$ and $ \mid \Cent(G_2)\mid=8$. In view of \cite[Lemma 3.1]{rahul1}, $G_2$ is isoclinic with a finite $2$-group and hence $\mid G_2' \mid \neq 6$. From the above result, we can also see that if $G$ and $H$ are $n$-centralizer groups with $n \in \lbrace 4, 5, 7,  9\rbrace$, then $G' \cong H'$. However, $D_{16}$ and $A_4$ are $6$-centralizer groups with $D_{16}' \cong C_4$ and $A_4' \cong C_2 \times C_2$.

\begin{prop}\label{p212}
Let  $G$ and $H$ be two finite groups of conjugate type $(p, 1)$, $p$ a prime. 
If $\mid \Cent(G)\mid=\mid \Cent(H)\mid$, then $G$ is isoclinic with $H$.  
\end{prop}

\begin{proof}
In view of Ito \cite{ito}, without any loss we may assume that $G$ is a $p$-group. Furthermore, using Theorem \ref{ff1},  $G$ is isoclinic with a group $G_1$ such that $Z(G_1) \subseteq G_1'$. Note that since $Z(G_1)$ is finite, therefore $G_1$ is finite. It now follows using Theorem \ref{ff155} and \cite[Proposition 3.1]{ish1} that $G_1$ is an extraspecial $p$-group. Similarly, we can see that $H$ is isoclinic with an extraspecial $p$-group $G_2$. Now, suppose $\mid \Cent(G)\mid=\mid \Cent(H)\mid$. In the present scenario, using Proposition \ref{b439}, we have $\mid \Cent(G_1)\mid=\mid \Cent(G_2)\mid$ and consequently, applying \cite[Proposition 3.13]{baishyaF} if follows that $\mid G_1 \mid=\mid G_2 \mid$. Now the result follows using  Remark \ref{rem111}.  
\end{proof}

\begin{prop}\label{p21}
Let  $G$ be an extraspecial $p$-group of order $p^k$ for some $k$ and prime $p$. 
If $H$ is a subgroup of $G$ such that $\mid \Cent(G)\mid=\mid \Cent(H)\mid$, then $G=H$.  
\end{prop}

\begin{proof}
Since $\mid \Cent(G)\mid=\mid \Cent(H)\mid$, therefore applying Lemma \ref{b0}, we have  $H \cap Z(G)=Z(H)=Z(G)$ and consequently, $\frac{H}{Z(H)}=\frac{H}{Z(G) \cap H} \cong \frac{HZ(G)}{Z(G)} \leq \frac{G}{Z(G)}$. It now follows that $H$ is  an extraspecial $p$-group of order $p^l$ for some $l$. Therefore by \cite[Proposition 3.13]{baishyaF}, $k=l$ and hence $G=H$.
\end{proof}

\begin{prop}\label{per}
Let  $G$ be any group such that $\frac{G}{Z(G)} \cong C_p \times C_p$, $p$ a prime. Then $G$ is isoclinic with an extraspecial group of order $p^3$.
\end{prop}

\begin{proof}
Using Theorem \ref{ff1} and arguments in the proof of Theorem \ref{ff3}, $G$ is isoclinic with a finite group $N$ of order $p^n$ with $Z(N) \subseteq N'$. Moreover, since $\frac{N}{Z(N)} \cong C_p \times C_p$, therefore $Z(N) = N'$. Also note that any proper centralizer of $N$ is abelian normal of index $p$ in $N$. Hence using  \cite[Lemma 4, pp. 303]{zumud}, we have $\mid N \mid=p.\mid Z(N) \mid .\mid N'\mid$ and consequently, $N$ is an extraspecial
group of order $p^3$. 
\end{proof}

\begin{cor}\label{cor444}
Let  $G$ and $H$ be any two groups such that $\frac{G}{Z(G)} \cong \frac{H}{Z(H)} \cong C_p \times C_p$, $p$ a prime. Then $G$ is isoclinic with $H$.
\end{cor}

\begin{proof}
The result follows from Proposition \ref{per} and Remark \ref{rem111}.
\end{proof}

The following result shows that any two $4$-centralizer groups are isoclinic.

\begin{prop}\label{corollary341}
Any $4$-centralizer group is isoclinic with $Q_8$.
\end{prop}

\begin{proof}
It follows using \cite[Theorem 3.5]{non}, Proposition \ref{per} and Remark \ref{rem111}.
\end{proof}

\begin{lem}\label{p1}
If $M$ is a maximal non-abelian subgroup of a CA-group $G$, then either $Z(G)=Z(M)$ or $G$ is isoclinic with $M$.
\end{lem}

\begin{proof}
If $Z(G) \subseteq M$, then using Proposition \ref{b011} we have $Z(G)=Z(M)$. On the otherhand, if $Z(G) \nsubseteq M$, then $MZ(G)=G$ and hence $G$ is isoclinic with $M$ by \cite[Lemma 2.7]{pL95}. 
\end{proof}

\begin{prop}\label{ba567}
Let $H$ be a non-abelian subgroup of $G$ with $\mid \frac{G}{Z(G)} \mid=p^3$ ($p$ a prime).Then  $\mid \Cent(G)\mid=\mid \Cent(H)\mid$ implies  $G$ is isoclinic with $H$.
\end{prop}

\begin{proof}
In view of \cite[Lemma 2.1]{baishya2} and  Proposition \ref{b011},  $\frac{H}{Z(H)} \cong \frac{HZ(G)}{Z(G)} \leq \frac{G}{Z(G)}$. In the present scenario using \cite[Proposition 2.14]{baishya2}, we have $HZ(G)=G$ by noting that if $\mid \frac{H}{Z(H)} \mid=p^2$, then $\mid \Cent(H)\mid=p+2$. Hence by \cite[Lemma 2.7]{pL95}, $G$ is isoclinic with $H$.
\end{proof}

Combining Corollary \ref{cor444} and \cite{non, baishya1}  it is easy to see that any two nilpotent $n \in \lbrace 5, 7, 9 \rbrace$-centralizer  groups are isoclinic. For $5$-centralizer groups we have the following result:

\begin{prop}\label{corollary1}
Any $5$-centralizer group $G$ is isoclinic with $G_m= \langle a, b \mid a^3=b^{2^m}=1, bab^{-1}=a^{-1}\rangle$, where $m \geq 1$ or an extraspecial group of order $27$.
\end{prop}

\begin{proof}
In view of Theorem \ref{ff3} without any loss we may assume that $G$ is finite. Moreover, by  \cite[Theorem 3.5]{non} we have $\frac{G}{Z(G)} \cong S_3$ or $ C_3 \times C_3$. Now, if   $\frac{G}{Z(G)} \cong S_3$, then using \cite[Corollary 2.2]{lescot}, we have $G=G_m \times A$, where $m \geq 1$ and $A$ is an abelian group. Hence by Remark \ref{rem111}, $G$ is isoclinic with $G_m$. Again, if  $\frac{G}{Z(G)} \cong C_3 \times C_3$, then by Proposition \ref{per}, we have the result.
\end{proof}

Let $p$ be a prime. The author in \cite[Theorem 3.3]{en09} proved that if $G$ is a finite $(p^2+2)$-centralizer group of conjugate type $(p^2, 1)$  and two of the proper element centralizers are normal in $G$, then $\frac{G}{Z(G)}$ is elementary abelian of order $p^4$. We conclude the paper with the following generalization of this result.

\begin{thm}\label{b7}
Let $G$ be any finite $(n+2)$-centralizer   group of conjugate type $(n, 1)$. Then $G$ is a CA-group and $\frac{G}{Z(G)}$ is elementary abelian of order $n^2$.
\end{thm}

\begin{proof}
In view of Ito \cite{ito}, without any loss we may assume that $G$ is a $p$-group for some prime $p$. Let $X_i=C(x_i), 1\leq i \leq n+1$ where $x_i \in G \setminus Z(G)$. We have $G=\overset{n+1}{\underset{i=1}{\cup}}X_i$ and  $\mid G \mid=\overset{n+1}{\underset{i=2}{\sum}} \mid X_i \mid$. In the present scenario, interchanging $X_i$'s and applying  \cite[Cohn's Theorem]{cohn}, we have $G=X_iX_j$ and $X_i \cap X_j=Z(G)$ for any $1 \leq i, j \leq n+1, i \neq j$. It is easy to verify that $\mid \frac{G}{Z(G)} \mid=n^2$ and $G$ is a CA-group. Moreover, using \cite [Proposition 2]{mann} we have $\frac{G}{Z(G)}$ is elementary abelian.
\end{proof}

\section*{Acknowledgment}

I would like to thank Prof. Mohammad Zarrin for carefully reading the manuscript and giving his valuable suggestions and comments on it.


\begin{thebibliography}{33}

\bibitem{ed09}
A. Abdollahi, S. M. J. Amiri and A. M. Hassanabadi,  {\em Groups with specific number of centralizers}, Houst. J. Math., \textbf{33} (1) (2007), 43--57.


\bibitem{amiri2019}
 S. M. J. Amiri, H. Madadi and H. Rostami,  {\em On $9$-centralizer groups},  J. Algebra Appl., \textbf{14} (1) (2015), 01--13.


\bibitem{rostami1}
 S. M. J. Amiri and H. Rostami,  {\em Groups with a few non-abelian centralizers},  Publ. Math. Debrecen, \textbf{87}   (3--4) (2015), 429--437.



\bibitem{en09}
 A. R. Ashrafi, {\em On finite groups with a given number of centralizers}, Algebra Colloq., \textbf{7} (2) (2000), 139--146.
 
 

  
  
\bibitem{taeri}
A. R. Ashrafi and B. Taeri, {\em On finite groups with a certain number of centralizers},  J. Appl. Math. Comput., \textbf{17} (2005), 217--227.
  
  
\bibitem{actc09}
A. R. Ashrafi and B. Taeri, {\em On finite groups with exactly seven element centralizers},  J. Appl. Math. Comput., \textbf{22} (1--2) (2006), 403--410.
   
    
\bibitem{ashrafinew}
A. R. Ashrafi, F. K. Moftakhar and M. A. salahshour, {\em Counting the number of centralizers of $2$-element subsets in a finite group}, Comm. Alg., \textbf{48} (11) (2020) 4647--4662.
  
 

\bibitem{baishya}
S. J. Baishya, {\em On  finite groups with specific number of centralizers}, Int. Elect. J. Algebra, \textbf{13} (2013) 53--62.

\bibitem{baishya1}
S. J. Baishya, {\em On  finite groups with nine centralizers}, Boll. Unione Mat. Ital., \textbf{9} (2016) 527--531.



\bibitem{baishya2}
S. J. Baishya, {\em On  capable groups of order $p^2q$}, Comm. Alg., \textbf{48} (6)  (2020), 2632--2638.




\bibitem{baishyaF}
S. J. Baishya, {\em Counting centralizers and $z$-classes of some F-groups}, Comm. Alg., \textbf{50} (6) (2022) 2476--2487.

\bibitem{baishyarima}
S. J. Baishya, {\em Characterizations of some groups in terms of centralizers}, Results Math, \textbf{77, 168} (2022) https://doi.org 10.1007/s00025-022-01687-4.


\bibitem{ctc092}
S. M. Belcastro and G. J. Sherman, {\em Counting centralizers in finite groups}, Math. Magazine, \textbf{67} (5) (1994), 366--374.

\bibitem{zumud}
Y. G. Berkovich and   E. M. Zhmud$^{\prime}$,   \textit{Characters of Finite Groups}. Part 1, Transl. Math. Monographs {\bf 172}, Amer. Math. Soc., Providence, RI, 1998. 




\bibitem{cohn}
J. H. E. Cohn, {\em On $n$-sum groups}, Math. Scand, \textbf{75} (1994), 44--58.



\bibitem{conrad}
K. Conrad, {\em Dihedral groups II}, https://kconrd.math.uconn.edu/Expository papers.


\bibitem{extcent}
M. Divandari, F. P. Shanbehbazari, P. Niroomand and A. F. Salles, {\em On finite groups with a given njumber of exterior centralizers}, Comm. Alg., \textbf{47} (1) (2019) 182--187.




\bibitem{hall}
P. Hall, {\em The classification of prime power groups},  J. reine angew Math.,  \textbf{182} (1940), 130--141.


\bibitem{nca}
M. A. Iranmanesh and M. H. Zareian, {\em On $n$-centralizer CA-groups}, Comm. Alg., \textbf{49} (10) (2021) 4186--4195.


\bibitem{isaacs}
I. M. Isaacs, {\em Character Theory of Finite groups}, Dover Publications, Inc., New York, (1994). 


\bibitem{ish1}
K. Ishikawa, {\em Finite $p$-groups upto isoclinism, Which have only two conjugacy lengths},  J. Algebra,  \textbf{220} (1999), 333--345.


\bibitem{ito}
N. Ito, {\em On finite groups with given conjugate type, I}, Nagoya J. Math.,  \textbf{6} (1953), 17--28.


\bibitem{con}
K. Khoramshahi and M. Zarrin, {\em Groups with the same number of centralizers}, J. Algebra Appl. \textbf{20} (2) (2021), 6 pages.


\bibitem{rahul2}
R. D. Kitture, {\em Groups with finitely many centralizers}, Bull. Allahabad Math. Soc., \textbf{30}  (2015), 29--37.

\bibitem{rahul1}
R. Kulkarni, R. D. Kitture and V. S. Yadav, {\em $z$-classes in groups}, J. Algebra Appl., \textbf{15} (2016), 01--16.



\bibitem{pL95}
P. Lescot, {\em  Isoclinism classes and commutativity degrees of finite groups},  J. of Algebra, {\bf 177}
(1995), 847--869.

\bibitem{lescot}
P. Lescot, {\em  Central extensions and commutativity degree}, Comm. Algebra {\bf 29} (10) (2001), 4451–4460.

\bibitem{lewis}
M. L. Lewis, {\em Semi extraspecial groups}, Advances in algebra, Springer Proc. Math. Stat., \textbf{277}, Springer, Cham,  (2019), 219--237.


\bibitem{mann}
A. Mann, {\em Extreme elements of finite $p$-groups}, Rend. Sem. Mat. Univ. Padova,  \textbf{83} (1990), 45--54.



\bibitem{zarrin094}
M. Zarrin, {\em On element centralizers in finite groups}, Arch. Math.,  \textbf{93} (2009), 497--503.





\bibitem{non}
M. Zarrin, {\em On non-commuting sets and centralizers in infinite groups}, Bull. Aust. Math. Soc. \textbf{93} (2016), 42--46.


\end{thebibliography}
\end{document}